\documentclass[a4paper,12pt]{article}
\usepackage[utf8]{inputenc}
\usepackage{amsmath}
\usepackage{amssymb}
\usepackage{amsthm}

\usepackage{cite}

\usepackage{booktabs}

\usepackage{esint}

\usepackage[scr]{rsfso}
\usepackage{bbm}

\usepackage{color}

\makeatletter
\let\my@saved@original@eqref\eqref 
\renewcommand*{\eqref}[1]{
  \begingroup
    \let\normalfont\relax
    \my@saved@original@eqref{#1}
  \endgroup
}
\makeatother
\makeatletter
\renewcommand*\env@matrix[1][r]{\hskip -\arraycolsep
  \let\@ifnextchar\new@ifnextchar
  \array{*\c@MaxMatrixCols #1}}
\makeatother

\newcommand{\pt}{\partial}
\DeclareMathOperator{\Div}{{div}}

\DeclareMathOperator{\Curl}{{curl}}

\newcommand{\norm}[2]{\|{#1}\|_{#2}}

\newcommand{\N}{\mathbb{N}}

\newcommand{\R}{\mathbb{R}}

\newcommand{\PS}{\mathcal{P}}
\newcommand{\QS}{\mathcal{Q}}

\newcommand{\lrarrow}{\quad\Leftrightarrow\quad}

\newcommand{\qmbox}[1]{\quad\mbox{#1}\quad}

\newcounter{tmp}
\newcommand{\makeballnumber}[1]{\setcounter{tmp}{\theenumi}%
\setcounter{enumi}{#1}%
\leavevmode \csname beamer@@tmpl@enumerate item\endcsname%
\setcounter{enumi}{\thetmp}}
\newcommand{\makeball}{\leavevmode \csname beamer@@tmpl@itemize item\endcsname}

\definecolor{seb}{rgb}{0.9,0,0}

\newcommand{\vecsymb}[1]{\boldsymbol{#1}}

\newcommand{\vh}{\vecsymb{h}}

\newcommand{\vn}{\vecsymb{n}}

\newcommand{\vq}{\vecsymb{q}}

\newcommand{\vu}{\vecsymb{u}}
\newcommand{\vv}{\vecsymb{v}}

\newcommand{\valpha}{\vecsymb{\alpha}}

\newcommand{\vI}{\vecsymb{I}}

\title{Anisotropic $H_{\Div}$-norm error estimates for rectangular $H_{\Div}$-elements}
\author{Sebastian Franz\footnote{
          Institute of Scientific Computing, Technische Universit\"at Dresden, Germany.\newline
          \mbox{e-mail}: sebastian.franz@tu-dresden.de}
       }
\date{\today}

\usepackage{fancyhdr} 
\theoremstyle{plain}
\newtheorem{theorem}{Theorem}[section]
\newtheorem{lemma}[theorem]{Lemma}

\renewcommand{\vI}{\boldsymbol{\mathcal{I}}}

\begin{document}
  \pagestyle{fancy}
  \maketitle
  \begin{abstract}
    For the discretisation of $H_{\Div}$-functions on rectangular meshes there are at least three families of
    finite elements, namely Raviart-Thomas-, Brezzi-Douglas-Marini- and Arnold-Boffi-Falk-elements. 
    In order to prove convergence of a numerical method using them, sharp interpolation error estimates are important. 
    We provide them here in an anisotropic setting for the $H_{\Div}$-norm.
  \end{abstract}

  \textit{AMS subject classification (2010):} 65D05, 65N30\\

  \textit{Key words: $H_{\Div}$-elements, anisotropic interpolation error estimates, $H_{\Div}$-norm} \\

  \section{Introduction}
  The discretisation of $H_{\Div}$ using piecewise polynomials is known for a long time, see e.g. \cite{RT77}
  for one of the first papers, and
  they are used to discretise a variety of problems. For the analysis of these numerical methods we need
  estimates of the interpolation error. Especially on anisotropic meshes a finely tuned estimate, that incorporates the anisotropy, is important.
  In \cite{AADL11} anisotropic $L^p$-interpolation error estimates for Raviart-Thomas elements and the interpolation operator $\vI$
  on simplicial meshes
  \begin{gather}\label{eq:anisoL2_1}
    \norm{\vu-\vI\vu}{L^p(T)}\lesssim \sum_{|\valpha|=k+1}\vh^{\valpha}\norm{\pt_x^{\alpha_1}\pt_y^{\alpha_2}\vu}{L^p(K)}+h_T^{k+1}\norm{D^{k+1}\Div\vu}{L^p(T)}
  \end{gather}
  is presented, where ${\valpha}$ is a multiindex, $\vh$ are the lengths of an anisotropic simplex $T$, $h_T$ its diameter and $D^k$ denotes the sum of the absolute values of all the 
  derivatives of order $k$. 
  A similar result is given in \cite{AK20} for the Brezzi-Douglas-Marini element.
  They are not completely anisotropic due to the final term which for simplices and $\Div \vu\neq 0$ cannot be neglected.
  
  On rectangles and more general $d$-dimensional parallelotopes the situation is different. Here \cite{Stynes14} provides
  several anisotropic interpolation error estimates for the $L^p$-norm error based on Poincar\'e-inequalities.
  One of optimal order with minimal assumptions on the regularity of $\vu$ is 
  \begin{gather}\label{eq:anisoL2_2}
    \norm{\vu-\vI\vu}{L^p(K)}\lesssim \sum_{|\valpha|={k+1}}\vh^{\valpha}\norm{\pt_x^{\alpha_1}\pt_y^{\alpha_2}\vu}{L^p(K)}.
  \end{gather}
  Compared with estimate \eqref{eq:anisoL2_1} this estimate is completely anisotropic.

  In all these publications no anisotropic estimate for $\norm{\Div(\vu-\vI\vu)}{L^2(\Omega)}$ is given. We will present here such an estimate for rectangular
  elements in 2d. The generalisation to 3d and beyond is straight forward.
    
  \textbf{Notation:} We denote vector valued functions with a bold font. $L^p(D)$ with the norm $\norm{\cdot}{L^p(D)}$
  is the classical Lebesque space of function integrable to the power $p$ over a domain $D\subset\R^2$ and $W^{\ell,p}(D)$
  the corresponding Sobolev space of (weak) derivatives up to order $\ell$. Furthermore, we write $A\lesssim B$ if there exists 
  a generic constant $C>0$ such that $A\leq C\cdot B$.
  
  \section{Interpolation error estimates}
  Let us denote by $\hat K:=[0,1]^2$ the reference square, by $\QS_{p,q}(\hat K)$ the space of
  polynomials with degree $p$ and $q$ in the two dimensions over $\hat K$ and 
  $\QS_k(\hat K):=\QS_{k,k}(\hat K)$. Furthermore, let $\PS_k(\hat K)$ be the space of polynomials of total degree $k$ over $\hat K$.
  
  The basic ingredient of an interpolation error estimate for $\Div\vu$ is the commuting diagram property.
  Let $V_h$ be a discrete space over $\hat K$, $\hat\vI$ an interpolator into $V_h$ and $\hat\Pi$ the $L^2$-projection into $\Div V_h$.
  Then the commuting diagram property is, see \cite[Remark 2.5.2]{BBF13}
  \begin{gather}\label{eq:commdiag}
    \Div \hat\vI \vu=\hat\Pi\Div\vu.
  \end{gather}
  We will use this property first in an abstract way and apply it to the three families of 
  finite elements afterwards.
  
  Having \eqref{eq:commdiag} we obtain for $w:=\Div \vu$
  \[
    \norm{\Div(\vu-\hat\vI\vu)}{L^p(\hat K)}
      =\norm{\Div\vu-\hat\Pi\Div\vu}{L^p(\hat K)}
      =\norm{w-\hat\Pi w}{L^p(\hat K)},
  \]
  and the estimation becomes that of the $L^2$-projection into $\Div V_h$. A very useful tool in proving 
  anisotropic interpolation error estimates is the technique presented in \cite{Apel99}. For it to apply we only need 
  a set of functionals with some properties.
  \begin{lemma}\label{lem:functionals}
    There exist $d:=\dim(\Div V_h)$ functionals $F_{i}$ and an integer $\ell\in\N_0$, such that for all $p\in[1,\infty)$ and $1\leq i\leq d$
    \begin{subequations}
      \begin{align}
        &F_{i}\in (W^{\ell,p}(\hat K))',\label{eq:functional_1}\\
        &F_{i}(\hat\Pi w-w)=0\text{ for all }w\in C(\hat K),\label{eq:functional_2}\\
        &w\in \Div V_h\text{ and }F_{i}(w)=0\text{ for all } i\in\{1,\dots,d\}\text{ implies }w=0.\label{eq:functional_3}
      \end{align}
    \end{subequations}
  \end{lemma}
  \begin{proof}
    Let us define for all $w\in W^{\ell,p}(\hat K)$ the linear functionals $F_{i}$ by
    \[
      F_{i}(w):=\int_{\hat K}w\cdot q_i,\quad i\in\{1,\dots,d\},
    \]
    where $\{q_i\}$ is a basis of $\Div V_h$. Note that these functionals can also be used in defining the $L^2$-projection $\hat\Pi$.
    Then it holds
    \[
      |F_{i}(w)|
        \leq \norm{q_i}{L^{q}(\hat K)}\norm{w}{L^p(\hat K)}
        \lesssim \norm{w}{W^{\ell,p}(\hat K)},
    \]
    where $\frac{1}{p}+\frac{1}{q}=1$. Thus we have \eqref{eq:functional_1}. 
    The consistency \eqref{eq:functional_2} follows directly by definition of the functionals and the $L^2$-projection.
    Finally, for $w\in \Div V_h$ we obtain
    \[
      F_{i}(w)=0,\,i\in\{1,\dots,d\}\quad 
      \lrarrow
      \hat\Pi w=0
      \lrarrow
      w=0,
    \]
    which is \eqref{eq:functional_3}.
  \end{proof}
  Following the technique of \cite{Apel99}, shown therein for Lagrange and Scott-Zhang interpolation,
  and using Lemma~\ref{lem:functionals}, we obtain the anisotropic interpolation error estimate for $\Div \vu$.
  In the case of $\PS_\ell(\hat K)\subset\Div V_h$  and assuming $\Div\vu\in W^{\ell+1,p}(\hat K)$
  it can be written as
  \begin{align}
    \norm{\Div(\vu-\hat\vI\vu)}{L^p(\hat K)}
    &=\norm{\Div\vu-\hat\Pi\Div\vu}{L^p(\hat K)}
    \lesssim \sum_{|\valpha|=\ell+1}\norm{\pt_x^{\alpha_1}\pt_y^{\alpha_2}\Div\vu}{L^p(\hat K)},\label{eq:est:1}
  \intertext{where ${\valpha}$ is a multiindex, and for $\QS_\ell(\hat K)\subset \Div V_h$ we have the sharper estimate}
    \norm{\Div(\vu-\hat\vI\vu)}{L^p(\hat K)}
    &\lesssim \norm{\pt_x^{\ell+1}\Div\vu}{L^p(\hat K)}+\norm{\pt_y^{\ell+1}\Div\vu}{L^p(\hat K)}.\label{eq:est:2}
  \end{align}
  
  \subsection{Raviart-Thomas elements}
  The Raviart-Thomas space over $\hat K$ is given by, see \cite{BBF13},
  \[
    RT_k(\hat K):= \QS_{k+1,k}(\hat K)\times\QS_{k,k+1}(\hat K).
  \]
  Note, that it holds
  \[
    (\QS_{k}(\hat K))^2\subset RT_k(\hat K)\subset (\QS_{k+1}(\hat K))^2
    \qmbox{and}
    \Div RT_{k}(\hat K)=\QS_k(\hat K).
  \]  
  The interpolation operator $\hat\vI_{RT}: (C(\hat K))^2\to RT_k(\hat K)$ is given for any $\vv\in (C(\hat K))^2$ by
  \begin{subequations}\label{eq:interpolation}
  \begin{align}
    \int_{\hat F}(\hat\vI_{RT}\vv-\vv)\cdot \vn\cdot q&=0,\quad\forall q\in \PS_k(\hat F),\forall \hat F\subset\partial\hat K,\\
    \int_{\hat K}(\hat\vI_{RT}\vv-\vv)\cdot\vq&=0,\quad\forall \vq\in \QS_{k-1,k}(\hat K)\times\QS_{k,k-1}(\hat K),
  \end{align}
  \end{subequations}
  where $\PS_k(\hat F)$ is the space of polynomials of total degree $k$ on a face $\hat F$ of $\hat K$. 
  
  The commuting diagram property \eqref{eq:commdiag} can be shown with integrations by parts
  and using the properties of the interpolation operator, see also \cite[Proposition 2.5.2]{BBF13}.
  Thus we have everything needed to apply the general result \eqref{eq:est:2}.
  \begin{theorem}
    For any $1\leq\ell\leq k+1$, $p\in[1,\infty)$ and $\vu\in (L^1(\hat K))^2$, such that $\Div \vu\in W^{\ell,p}(\hat K)$ 
    it holds
    \[
      \norm{\Div(\vu-\hat\vI_{RT}\vu)}{L^p(\hat K)}
      \lesssim \norm{\pt_x^{\ell}\Div\vu}{L^p(\hat K)}+\norm{\pt_y^{\ell}\Div\vu}{L^p(\hat K)}.
    \]
  \end{theorem}
  
  \subsection{Brezzi-Douglas-Marini elements}
  The Brezzi-Douglas-Marini space over $\hat K$ is given by, see \cite{BDM85},
  \[
    BDM_k(\hat K):= (\PS_k(\hat K))^2\oplus\text{span}\{\Curl x^{k+1}y ,\Curl xy^{k+1}\}
  \]
  and it holds
  \[
    (\PS_{k}(\hat K))^2\subset BDM_k(\hat K)
    \qmbox{and}
    \Div BDM_k(\hat K)=\PS_{k-1}(\hat K).
  \]
  The interpolation operator $\hat\vI_{BDM}: (C(\hat K))^2\to BDM_k(\hat K)$ is given for any $\vv\in (C(\hat K))^2$ by
  \begin{subequations}\label{eq:interpolationBDM}
  \begin{align}
    \int_{\hat F}(\hat\vI_{BDM}\vv-\vv)\cdot \vn\cdot q&=0,\quad\forall q\in \PS_{k}(\hat F),\forall \hat F\subset\partial\hat K,\\
    \int_{\hat K}(\hat\vI_{BDM}\vv-\vv)\cdot\vq&=0,\quad\forall \vq\in (\PS_{k-2}(\hat K))^2,
  \end{align}
  \end{subequations}
  and again the commuting diagram property \eqref{eq:commdiag} can be shown with integrations by parts, see also \cite[Proposition 2.5.2]{BBF13}.
  We therefore obtain with \eqref{eq:est:1} the following theorem.
  \begin{theorem}
    For any $1\leq\ell\leq k$, $p\in[1,\infty)$ and $\vu\in (L^1(\hat K))^2$, such that $\Div \vu\in W^{\ell,p}(\hat K)$ 
    it holds
    \[
      \norm{\Div(\vu-\hat\vI_{BDM}\vu)}{L^p(\hat K)}
      \lesssim \sum_{|\valpha|=\ell}\norm{\pt_x^{\alpha_1}\pt_y^{\alpha_2}\Div\vu}{L^p(\hat K)},
    \]
    where ${\valpha}$ is a multiindex of degree $\ell$.
  \end{theorem}
  
  \subsection{Arnold-Boffi-Falk elements}
  The Arnold-Boffi-Falk space over $\hat K$ is given by, see \cite{ABF05}, 
  \[
    ABF_k(\hat K):= \QS_{k+2,k}(\hat K)\times\QS_{k,k+2}(\hat K)
  \]
  and it holds
  \[
    (\QS_{k}(\hat K))^2\subset ABF_k(\hat K)
    \qmbox{and}
    \Div ABF_k(\hat K)=\QS_{k+1}(\hat K)\setminus\text{span}\{x^{k+1}y^{k+1}\}.
  \]
  The interpolation operator $\hat\vI_{ABF}: (C(\hat K))^2\to ABF_k(\hat K)$ is given for any $\vv\in (C(\hat K))^2$ by
  \begin{subequations}\label{eq:interpolationABF}
  \begin{align}
    \int_{\hat F}(\hat\vI_{ABF}\vv-\vv)\cdot \vn\cdot q&=0,\quad\forall q\in \PS_{k}(\hat F),\forall \hat F\subset\partial\hat K,\\
    \int_{\hat K}(\hat\vI_{ABF}\vv-\vv)\cdot\vq&=0,\quad\forall \vq\in \QS_{k-1,k}(\hat K)\times\QS_{k,k-1}(\hat K),\\
    \int_{\hat K}\Div(\hat\vI_{ABF}\vv-\vv)\cdot x^iy^{k+1}&=0,\quad\forall i\in\{0,\dots,k\},\label{eq:ABF:c}\\
    \int_{\hat K}\Div(\hat\vI_{ABF}\vv-\vv)\cdot x^{k+1}y^j&=0,\quad\forall j\in\{0,\dots,k\},\label{eq:ABF:d}
  \end{align}
  \end{subequations}
  and again the commuting diagram property \eqref{eq:commdiag} can be shown with integrations by parts and 
  a direct application of \eqref{eq:ABF:c} and \eqref{eq:ABF:d}.
  Note that due to
  \[
    \PS_{k+1}(\hat K)\subset \Div ABF_k(\hat K)\qmbox{and}
    \QS_{k}(\hat K)\subset \Div ABF_k(\hat K)
  \]
  we can use both \eqref{eq:est:1} and \eqref{eq:est:2} to estimate the interpolation error.
  \begin{theorem}
    For any $1\leq\ell\leq k+2$, $p\in[1,\infty)$ and $\vu\in (L^1(\hat K))^2$, such that $\Div \vu\in W^{\ell,p}(\hat K)$ 
    it holds
    \begin{align*}
      \norm{\Div(\vu-\hat\vI_{ABF}\vu)}{L^p(\hat K)}
      &\lesssim \sum_{|\valpha|=\ell}\norm{\pt_x^{\alpha_1}\pt_y^{\alpha_2}\Div\vu}{L^p(\hat K)},
    \intertext{and for $1\leq s\leq k+1$}
      \norm{\Div(\vu-\hat\vI_{ABF}\vu)}{L^p(\hat K)}
      &\lesssim \norm{\pt_x^{s}\Div\vu}{L^p(\hat K)}+\norm{\pt_y^{s}\Div\vu}{L^p(\hat K)}.
    \end{align*}
  \end{theorem}
  
  \section{Conclusions}
    Let us transform the estimates shown on the reference element back to an axi-parallel 
    rectangle $K$ of dimensions $h_x$ and $h_y$.
    Then, always using the highest possible derivatives, we obtain for $p\in[1,\infty)$
    \begin{align*}
      \norm{\Div(\vu-\vI_{RT}\vu)}{L^p(K)}
        &\lesssim h_x^{k+1}\norm{\pt_x^{k+1}\Div \vu}{L^p(K)}+h_y^{k+1}\norm{\pt_y^{k+1}\Div \vu}{L^p(K)},\\
      \norm{\Div(\vu-\vI_{BDM}\vu)}{L^p(K)}
        &\lesssim \sum_{|\valpha|=k}h_x^{\alpha_1}h_y^{\alpha_2}\norm{\pt_x^{\alpha_1}\pt_y^{\alpha_2}\Div \vu}{L^p(K)},\\
      \norm{\Div(\vu-\vI_{ABF}\vu)}{L^p(K)}
        &\lesssim \sum_{|\valpha|=k+2}h_x^{\alpha_1}h_y^{\alpha_2}\norm{\pt_x^{\alpha_1}\pt_y^{\alpha_2}\Div \vu}{L^p(K)},\\
      \norm{\Div(\vu-\vI_{ABF}\vu)}{L^p(K)}
        &\lesssim h_x^{k+1}\norm{\pt_x^{k+1}\Div \vu}{L^p(K)}+h_y^{k+1}\norm{\pt_y^{k+1}\Div \vu}{L^p(K)}.
    \end{align*}
    Thus for all elements we obtain anisotropic interpolation error estimates in $H_{\Div}$.
    Compared to Raviart-Thomas elements we see a reduction of one order for Brezzi-Douglas-Marini elements
    and an increase of one order for Arnold-Boffi-Falk elements, if all derivatives of $\Div \vu$ were used.
    For the later element we have also the same estimate as for Raviart-Thomas elements.
    
    For an affine transformation from $\hat K$ to a quadrilateral $K$ similar results follow. It is an open question, whether for non-affine
    transformations anisotropic interpolation error estimates can be shown.
    
    Using as functionals $F_i$ those conditions from the definition of the interpolation operator $\vI$,
    we can also show anisotropic interpolation error estimates in $L^p$, following \cite[Lemma 2.13]{Apel99}
    and improve upon \eqref{eq:anisoL2_2}
    \begin{align*}
      \norm{\vu-\vI_{RT}\vu}{L^p(K)}+\norm{\vu-\vI_{ABF}\vu}{L^p(K)}
        &\lesssim h_x^{k+1}\norm{\pt_x^{k+1}\vu}{L^p(K)}+h_y^{k+1}\norm{\pt_y^{k+1}\vu}{L^p(K)},\\
      \norm{\vu-\vI_{BDM}\vu}{L^p(K)}
        &\lesssim \sum_{|\valpha|=k+1}h_x^{\alpha_1}h_y^{\alpha_2}\norm{\pt_x^{\alpha_1}\pt_y^{\alpha_2}\vu}{L^p(K)}.
    \end{align*}

  \section*{Acknowledgements}
  The author would like to thank Thomas Apel and Gunar Matthies for valuable discussions
  concerning the interpolation of $H_{\Div}$-elements.
  \bibliographystyle{plain}
  \bibliography{lit}

\end{document}